\documentclass[a4paper,10pt,reqno]{amsart}

\title[Value Distribution of Two Dirichlet $L$-functions]{On the Value Distribution of Two Dirichlet $L$-functions}

\author{Niko Laaksonen}
\address{Department of Mathematics, University College London, Gower Street, London WC1E 6BT, United Kingdom}
\email{n.laaksonen@ucl.ac.uk}

\author{Yiannis N.~Petridis}
\address{Department of Mathematics, University College London, Gower Street, London WC1E 6BT, United Kingdom}
\email{i.petridis@ucl.ac.uk}
\thanks{The second author would like to thank ESI for support during the programme on Arithmetic Geometry and Automorphic Representations.}

\date{\today}
\subjclass[2010]{11M06, 11M26}
\keywords{Dirichlet $L$-function; value-distribution}

\usepackage{setspace} %This needs to be loaded BEFORE hyperref, otherwise get warnings with footnotes
\usepackage[backend=bibtex,bibencoding=ascii,style=numeric,firstinits=true]{biblatex}
\DeclareFieldFormat[article,inbook,incollection,inproceedings,patent,thesis,unpublished]{citetitle}{#1\isdot}
\DeclareFieldFormat[article,inbook,incollection,inproceedings,patent,thesis,unpublished]{title}{#1\isdot}

\renewcommand*{\intitlepunct}{\space}
\renewbibmacro{in:}{%
    \ifentrytype{article}{}{\printtext{\bibstring{in}\intitlepunct}}}
\DeclareFieldFormat[article]{volume}{\textbf{#1}}
\DeclareFieldFormat[article]{number}{\mkbibparens{#1}}
\renewbibmacro*{volume+number+eid}{%
    \printfield{volume}%
    \printfield{number}%
    \setunit{\addcomma\space}%
    \printfield{eid}}

\AtEveryBibitem{% Remove the following fields
    \clearfield{note}
    \clearfield{issn}
    \clearfield{url}
    \clearfield{bdsk-url-1}
    \clearfield{doi}
    \clearfield{isbn}
}

\usepackage{hyperref}
\hypersetup{%
	pdftitle={On the Value Distribution of Two Dirichlet L-functions},
	pdfauthor={Niko Laaksonen and Yiannis N. Petridis},
    pdfsubject={2010 Mathematics Subject Classification 11M06, 11M26}
}

\numberwithin{equation}{section}

\usepackage{amssymb}
\usepackage{mathtools} %for mathclap and substack

\usepackage[a4paper,margin=1.25in]{geometry}
\usepackage{enumerate}

\usepackage[utf8]{inputenc}
\usepackage[T1]{fontenc}

\usepackage{abstract}

\usepackage{mathrsfs}

\usepackage{xparse} %for more advanced macros

\newcommand{\primesP}{\mathscr{P}}

\newcommand{\naturals}{\mathbb{N}}
\newcommand{\reals}{\mathbb{R}}
\newcommand{\complexes}{\mathbb{C}}
\newcommand{\rationals}{\mathbb{Q}}

\DeclareMathOperator{\arccot}{arccot}

\newcommand{\tendsto}{\longrightarrow}
\newcommand{\bigo}[2][]{O_{#1}(#2)}
\newcommand{\lilo}[2][]{o_{#1}(#2)}
\newcommand{\bigos}[2][]{O_{#1}\left(#2\right)}

\newcommand{\ccontour}{\mathfrak{C}}
\newcommand{\inti}{\mathcal{I}}
\newcommand{\intf}{\mathcal{F}}
\newcommand{\inte}{\mathcal{E}}
\DeclareMathOperator*{\Res}{Res}
\DeclareMathOperator{\re}{Re}
\DeclareMathOperator{\im}{Im}
\DeclareMathOperator{\Arg}{Arg}

\DeclareDocumentCommand{\widesum}{m O{}}{\;\smashoperator{\sum_{\cramped{\substack{#1}}}^{#2}}\;}
\DeclareDocumentCommand{\rwidesum}{m O{}}{\smashoperator[r]{\sum_{\cramped{\substack{#1}}}^{#2}}\;} %to be used before parentheses, or stacking multiple widesums after each other
\DeclareDocumentCommand{\lwidesum}{m O{}}{\smashoperator[l]{\sum_{\cramped{\substack{#1}}}^{#2}}} %to be used before parentheses, or stacking multiple widesums after each other
\renewcommand{\d}[1]{\,d#1}
\let\abs\undefined
\DeclarePairedDelimiter{\abs}{\lvert}{\rvert}
\DeclarePairedDelimiter{\pdist}{\langle}{\rangle}
\providecommand\given{}
\newcommand{\SetSymbol}[1][]{\nonscript\:#1\vert\nonscript\:\mathopen{}\allowbreak}
\DeclarePairedDelimiterX\Set[1]\{\}{%
    \renewcommand\given{\SetSymbol[\delimsize]}
    #1
}

\theoremstyle{plain}
\newtheorem{thm}{Theorem}[section]
\newtheorem*{thm*}{Theorem}

\newtheorem{prop}{Proposition}[section]
\newtheorem{lemma}{Lemma}[section]
\theoremstyle{definition}

\theoremstyle{remark}
\newtheorem{remark}{Remark}

\newcommand{\qlt}{\sqrt{\frac{q\ell t}{2\pi}}}
\newcommand{\qtoverl}{\sqrt{\frac{qt}{2\pi \ell}}}
\newcommand{\ltoverq}{\sqrt{\frac{\ell t}{2\pi q}}}

\newcommand{\qlg}{\sqrt{\frac{q\ell\gamma}{2\pi}}}
\newcommand{\qgoverl}{\sqrt{\frac{q\gamma}{2\pi \ell}}}
\newcommand{\lgoverq}{\sqrt{\frac{\ell\gamma}{2\pi q}}}

\newcommand{\qlT}{\sqrt{\frac{q\ell T}{2\pi}}}

\newcommand{\qlTs}{q\ell T/2\pi}
\newcommand{\bige}{\varepsilon}
\newcommand{\af}{\mathfrak{a}}

\begin{document}

	\maketitle

	\begin{abstract}
        \noindent We look at the values of two Dirichlet $L$-functions at the Riemann zeros (or a horizontal shift of them). Off the
		critical line we show that for a positive proportion of these points the pairs of values of the two $L$-functions are linearly
		independent over $\reals$, which, in particular, means that their arguments are different. On the critical line we show that,
        up to height $T$, the values are different for $cT$ of the Riemann zeros for some positive $c$.
	\end{abstract}
    \vspace{\baselineskip}

% % % % % % % % % % % % % % % % % % % % %
%										%
%	§1	Intro							%
%										%
\section{Introduction}					%
% % % % % % % % % % % % % % % % % % % % %

    The value distribution of $\zeta(s)$ and $L(s,\chi)$ is a classical problem that has recently attracted attention
    in for example~\cite{garunkstis2010},~\cite{gonek1984},~\cite{steuding2007}. A typical focus for investigation for
    these functions is the distribution of their zeros.
	In 1976 Fujii~\cite{fujii1976} showed that a positive proportion of zeros of $L(s,\psi)L(s,\chi)$ are distinct, where
    the characters are primitive and not necessarily of distinct moduli. A zero of the product is said to be distinct
    if it is a zero of only one of the two, or if it is a zero of both then it occurs with different multiplicities for
    each function.
    It is, in fact, believed that all zeros of Dirichlet $L$-functions to primitive characters are simple, and that
    two $L$-functions with distinct primitive characters do not share any non-trivial zeros at all. This comes from
    the Grand Simplicity Hypothesis (GSH), see~\cite{rubinstein1994}. The hypothesis is that the set
    \[\Set{\gamma\given L(\tfrac{1}{2}+i\gamma,\chi)=0\text{ and $\chi$ is primitive}}\]
	is linearly independent over $\rationals$. Since we are counting with multiplicities, it is implicit in the
    statement of the GSH that all zeros of Dirichlet $L$-functions are simple, and that $\gamma\neq 0$,
    i.e.\ $L(\tfrac{1}{2},\chi)\neq 0$. A similar result is expected for an even bigger class of functions.
    R.~Murty and K.~Murty~\cite{murty1994} proved that two functions of the Selberg class $\mathcal{S}$ cannot share too many zeros
    (counted with multiplicity). They show that if $F$, $G\in\mathcal{S}$ then $F=G$ provided that
    \[\abs{Z_{F}(T)\Delta Z_{G}(T)}=o(T),\]
    where $Z_{F}(T)$ denotes the set of zeros of $F(s)$ in the region $\re s\geq 1/2$ and $\abs{\im s}\leq T$,
    and $\Delta$ is the symmetric difference. In 1986 Conrey et al.~\cite{conrey1986} proved that the Dedekind zeta function
    of a quadratic number field has infinitely many simple zeros. They were, however, unable to obtain
    the result for a positive proportion of the zeros.

    Apart from looking at the zeros, there has also been investigation into the $a$-values of $\zeta$ and $L(s,\chi)$, that is,
    the distribution of $s$ such that $\zeta(s)=a$ (or $L(s,\chi)=a$) for some fixed $a\in\complexes$. In~\cite{garunkstis2014} Garunkštis and Steuding prove
    a discrete average for $\zeta'$ over the $a$-values of $\zeta$, which implies that there are infinitely many simple $a$-points in the critical
    strip. On the critical line, however, we do not even know whether there are infinitely many $a$-points. For further results on the distribution
    of simple $a$-points see~\cite{gonek2012}.
    On the other hand, we can also look at points where $\zeta(s)$ (or $L(s,\chi)$) has a specific fixed argument $\varphi\in(-\pi,\pi]$. In~\cite{kalpokas2013}
    the authors prove that $\zeta$ takes arbitrarily large values with argument $\varphi$, i.e.
    \[\smashoperator{\max_{\substack{0<t\leq T\\ \Arg(\zeta(1/2+it))=\varphi}}}\;\abs{\zeta(\tfrac{1}{2}+it)}\gg(\log T)^{5/4}.\]

	In this work we compare the values or the arguments of
    two Dirichlet $L$-functions at a specific set of sample points. We choose these points to be either the Riemann zeros, $\beta+i\gamma$,
    or a horizontal shift of them.
    We will prove two results in this direction depending on whether we are on the critical line or not.
    \begin{thm}\label{thmone}
        Assume the Riemann Hypothesis, i.e.~$\beta=\tfrac{1}{2}$. Let $\chi_1$, $\chi_2$ be two primitive Dirichlet characters modulo distinct primes $q$ and $\ell$, respectively.
        Let $\sigma\in(\frac{1}{2},1)$, then, for a positive proportion of the non-trivial zeros of $\zeta(s)$ with $\gamma>0$, the values of the Dirichlet $L$-functions $L(\sigma+i\gamma,\chi_{1})$ and $L(\sigma+i\gamma,\chi_{2})$ are linearly independent over $\reals$.
    \end{thm}
	\begin{remark}
		If the values $L(\sigma+i\gamma,\chi_{1})$ and $L(\sigma+i\gamma,\chi_{2})$ are linearly independent over $\reals$, then in particular their arguments are different.
	\end{remark}
    \begin{thm}\label{thmtwo}
        Two Dirichlet $L$-functions with primitive characters modulo distinct primes, attain different values at $cT$ non-trivial zeros of $\zeta(s)$ up to height $T$, for some positive constant $c$.
    \end{thm}
	\begin{remark}
        In Theorem~\ref{thmtwo} we fail to obtain positive proportion and we expect this to be a limitation of the method used. In~\cite{garunkstis2013} the authors
        look at the mean square of a single Dirichlet $L$-function at the zeros of another, and show that it is non-zero for at least $cT$ of the zeros for some explicit $c>0$.
        On the other hand, attempting to introduce a mollifier to overcome this limitation does not seem hopefuly either. Martin and Ng~\cite{martin2013} evaluate the
        mollified first and second moments of $L(s,\chi)$ in arithmetic progressions on the critical line and prove that at least $T(\log T)^{-1}$ of the values are nonzero,
        which misses the positive proportion by a logarithm. This was extended to positive proportion by Li and Radziwiłł~\cite{li2015}. However, their method relies
        on the strong rigidity of the arithmetic progression and fails when the sequence is slightly perturbed.
	\end{remark}
	\begin{remark}
		We assume that the conductors of $\chi_{1}$ and $\chi_{2}$ are primes in order to make the notation simpler. It should be possible to generalise our results to
        the case when the conductors are coprime or have distinct prime factors.
	\end{remark}
    A main ingredient in the proofs is the Gonek--Landau formula, and results derived from it. In 1911 Landau~\cite{landau1911}
    proved that
    \[ \widesum{0<\gamma\leq T}x^{\rho}=-\frac{T}{2\pi}\Lambda(x)+\bigo{\log T},\]
    as $T\tendsto\infty$, where $\Lambda(x)$ is the von Mangoldt function extended to $\reals$ by letting $\Lambda(x)=0$
    $\forall x\in\reals\setminus\naturals$. Here the sum runs over the positive imaginary parts of the Riemann zeros.
    What is striking in this formula is that the right-hand side grows by a factor of $T$ only if $x$ is a prime power.
    This version of Landau's formula is of limited practical use since the estimate is not uniform in $x$.
    Gonek~\cite{gonek1993} proved a version of Landau's formula which is uniform in both $x$ and $T$ with only small
    sacrifices to the error term:
    \begin{lemma}[Gonek--Landau Formula]\label{thm:gonek}
        Let $x$, $T>1$. Then
        \begin{equation}\label{eq:landau}
            \begin{multlined}
            \widesum{0<\gamma\leq T}x^{\rho} =-\frac{T}{2\pi}\Lambda(x)+\bigos{x\log 2xT \,\log\log 3x}\\
            +\bigos{\log x\,\min\left(T,\frac{x}{\pdist{x}}\right)}+\bigos{\log 2T\,\min\left(T,\frac{1}{\log x}\right)},
            \end{multlined}
        \end{equation}
        where $\pdist{x}$ denotes the distance from $x$ to the nearest prime power other than $x$ itself.
    \end{lemma}
    If one fixes $x$ then this reduces to the original result of Landau as $T\tendsto\infty$.
    As an application of this result Gonek proves (under the RH) the following mean value for $\zeta$:
    \[ \widesum{0<\gamma\leq T}\abs{\zeta(\tfrac{1}{2}+i(\gamma+2\pi\alpha/\log T))}^{2}=\left(1-\left(\frac{\sin\pi\alpha}{\pi\alpha}\right)^{2}\right)\frac{T}{2\pi}\log^{2}T+\bigo{T\log^{7/4} T},\]
    where $T$ is large and $\alpha$ is real with $\lvert\alpha\rvert\leq\frac{1}{2\pi}\log T$.

    For Theorem~\ref{thmtwo} we need different tools. Since we are working with integrals the Gonek--Landau
    formula is not useful anymore. Instead we use a modified version, see~\cite{garunkstis2010}.
    \begin{lemma}[Modified Gonek Lemma]\label{modgonek}
        Suppose that $\sum_{n=1}^{\infty}a(n)n^{-s}$ converges for $\sigma>1$ and $a(n)=\bigo{n^{\epsilon}}$. Let $a=1+\log^{-1}T$. Then
        \begin{multline*}
                \frac{1}{2\pi i}\int_{a+i}^{a+iT}\left(\frac{m}{2\pi}\right)^{s}\Gamma(s)\exp\left(\delta\frac{\pi is}{2}\right)\sum_{n=1}^{\infty}\frac{a(n)}{n^{s}}\d{s}\\
                =\begin{cases}
                    \displaystyle\rwidesum{n\leq\frac{Tm}{2\pi}}a(n)\exp\left(-2\pi i\frac{n}{m}\right)+\bigo{m^{a}T^{1/2+\epsilon}},& \text{if } \delta = -1,\\
                    \bigo{m^{a}}, &	\text{if } \delta = +1.
                \end{cases}
        \end{multline*}
    \end{lemma}
    We also need the following version of the approximate functional equation for Dirichlet $L$-functions.
    First, denote by $G(k,\chi)$ the Gauss sum
    \[G(k,\chi)=\sum_{a=1}^{q}\chi(a)e^{2\pi iak/q}.\]
    We also write $G(1,\chi)=G(\chi)$.
    \begin{thm}[Lavrik~\cite{lavrik1968}]\label{thm:approx}
        Let $\chi$ be a primitive character mod $q$. For $s=\sigma+it$ with $0<\sigma<1$, $t>0$, and $x=\Delta\sqrt{\frac{qt}{2\pi}}$, $y=\Delta^{-1}\sqrt{\frac{qt}{2\pi}}$, and $\Delta\geq 1$, $\Delta\in\naturals$, we have
        \begin{equation}\label{eq:lapprox}
            L(s,\chi)=\sum_{n\leq x}\frac{\chi(n)}{n^{s}}+\bige(\chi)\left(\frac{q}{\pi}\right)^{\frac{1}{2}-s}\frac{\Gamma\bigl(\frac{1-s+\af}{2}\bigr)}{\Gamma\left(\frac{s+\af}{2}\right)}\sum_{n\leq y}\frac{\overline{\chi}(n)}{n^{1-s}}+R_{xy},
        \end{equation}
        with
        \begin{equation*}
            R_{xy}\ll\sqrt{q}\left(y^{-\sigma}+x^{\sigma-1}(qt)^{1/2-\sigma}\right)\log{2t},
        \end{equation*}
        and in particular, for $x=y$,
        \begin{equation*}
            R\ll x^{-\sigma}\sqrt{q}\log{2t}.
        \end{equation*}
        Here $\bige(\chi)=q^{-1/2}i^{\af}G(1,\chi)$, and
        \[\af= \frac{1-\chi(-1)}{2}.\]
        %			On the critical line we have for $q\leq t(\log^{4}{2t})^{-1}$
        %			\begin{equation*}
        %				R\ll\left(\frac{q}{t}\right)^{1/4}\log{2t}\ll 1.
        %			\end{equation*}
    \end{thm}
    It is not hard to see that this formula is, in fact, valid for all real $\Delta\geq 1$.
    Approximate functional equations for imprimitive characters do exist, but they are more complicated.
    Therefore, we restrict our attention to primitive characters in Theorem~\ref{thmone}.

    % % % % % % % % % % % % % % % % % % % % %
    %										%
    %	§3	Statement and outline			%
    %										%
    \section{Proof of Theorem~\ref{thmone}} %
    % % % % % % % % % % % % % % % % % % % % %

    The proof will follow the steps of \cite[Theorem 2]{gonek1993} and \cite[Theorem 1.9]{petridis2000} for the Riemann
    zeta function and $\mathrm{GL}_{2}$ $L$-functions.

    Two non-zero complex numbers $z$ and $w$ are linearly independent over the reals is equivalent to the quotient $z/w$
    being non-real, or that $\lvert z\bar{w}-\bar{z}w\rvert>0$. For us $z$ and $w$ are values of Dirichlet $L$-functions.
    Instead of looking at these functions at a single point, we will average over multiple points with a fixed real part
    $\sigma\in(\tfrac{1}{2},1)$ and the imaginary part at the height of the Riemann zeros.

    We are assuming the RH purely because it makes the proof simpler as expressions of the form $x^{\rho}$ become easier
    to deal with if we know the real part explicitly. On the other hand, the distribution of these specific points does
    not seem to have any impact on the rest of the proof. We suspect that the RH is not an essential requirement.
    In fact, following~\cite{gonek1984}, it might be possible to obtain the result without the RH by integrating
    \[ \frac{\zeta'}{\zeta}(s-\sigma)B(s,P)L(s,\chi_{1})\overline{L(s,\chi_{2})} \]
    over a suitable contour. This picks the desired points as residues of the integrand yielding the required sum.
    This idea is also used in the proof of Theorem~2.

    The proof will be divided into three propositions after which the main result follows easily.
    In the first proposition we want to calculate discrete mean values of sums of terms of the type
    $L(\sigma+i\gamma,\chi_{1})\overline{L(\sigma+i\gamma,\chi_{2})}$ and its complex conjugate. If we subtract one of
    these mean values from the other then each term is non-zero precisely when the two numbers are linearly independent
    over the reals. Hence we need to prove that the two mean values are not equal, which is the content of
    Proposition~\ref{prop3}. Finally, we get the main result by applying the Cauchy--Schwarz inequality to the difference
    of the mean values. Because of this we also need to estimate a sum of squares of the absolute values of the above
    quantities, that is,
    \[\abs{L(\sigma+i\gamma,\chi_{1})\overline{L(\sigma+i\gamma,\chi_{2})}-\overline{L(\sigma+i\gamma,\chi_{1})}L(\sigma+i\gamma,\chi_{2})}^{2}.\]
    This is done in Proposition~\ref{prop2}.

    The first problem in our proof is that the mean values are complex conjugates. In order to show that the difference
    is non-zero leads to determining whether $\im L(2\sigma,\chi_{1}\overline{\chi}_{2})\neq 0$, which does not always hold.
    Thus we need to introduce some kind of weighting in order to shove these sums off balance. We do this by multiplying
    by a finite Dirichlet polynomial, $B(s,P)$, which cancels some terms from either of the $L$-functions, depending on which
    mean value we are considering. We define
    \begin{equation}\label{eq:bsp}
        B(s,P)=\prod_{p\leq P}(1-\chi_{1}(p)p^{-s})(1-\chi_{2}(p)p^{-s})
    \end{equation}
    for some fixed prime $P$, depending only on $q$ and $\ell$, to be determined in Proposition~\ref{prop3}.
    Let us also assume that this Dirichlet polynomial has the expansion
    \[B(s,P)=\sum_{n\leq R}c_{n}n^{-s},\]
    for some $R$ depending on $P$. Since $\abs{c_{p}}\leq 2$ for any prime $p$, we have for all $n$ that
    \begin{equation}\label{eq:cestimate}
        \abs{c_{n}}\leq 2^{P}.
    \end{equation}
    We prove the following propositions.
    \begin{prop}\label{prop1}
        Assume the Riemann Hypothesis. With $s=\sigma+i\gamma$ we have
        \begin{equation}\label{prop1:eq1}
            \widesum{0<\gamma\leq T}B(s, P)L(s,\chi_{1})\overline{L(s,\chi_{2})}
            \sim N(T)\sum_{n=1}^{\infty}\frac{d_{n}\overline{\chi}_{2}(n)}{n^{2\sigma}},
        \end{equation}
        and
        \begin{equation}\label{prop1:eq2}
            \widesum{0<\gamma\leq T}B(s,P)\overline{L(s,\chi_{1})}L(s,\chi_{2})\sim N(T)\sum_{n=1}^{\infty}\frac{e_{n}\overline{\chi}_{1}(n)}{n^{2\sigma}},
        \end{equation}
        where
        \begin{equation*}
            B(s,P)L(s,\chi_{1})=\sum_{n=1}^{\infty}\frac{d_{n}}{n^{s}},\quad  B(s,P)L(s,\chi_{2})=\sum_{n=1}^{\infty}\frac{e_{n}}{n^{s}}.
        \end{equation*}
    \end{prop}

    \begin{prop}\label{prop2}
        Suppose $s=\sigma+i\gamma$ and let
        \begin{equation*}
            A(\gamma)=B(s,P)\bigl(L(s,\chi_{1})\overline{L(s,\chi_{2})}-\overline{L(s,\chi_{1})}L(s,\chi_{2})\bigr).
        \end{equation*}
        Then, under the Riemann Hypothesis,
        \begin{equation}\label{prop2:eq1}
            \widesum{0<\gamma\leq T}\abs{A(\gamma)}^{2}\ll N(T).
        \end{equation}
    \end{prop}

    \begin{prop}\label{prop3}
        Under the Riemann Hypothesis we can find a prime $P$ such that
        \begin{equation}\label{prop3:eq1}
            \widesum{0<\gamma\leq T} A(\gamma)\sim C\cdot N(T)
        \end{equation}
        for some non-zero constant $C$.
    \end{prop}

    \begin{proof}[Proof of Theorem~\ref{thmone}]
    By the Cauchy--Schwarz inequality and Propositions~\ref{prop2}~and~\ref{prop3}
    \begin{equation}\label{thmoneproof}
            \widesum{0<\gamma\leq T\\ A(\gamma)\neq 0}1\geq\frac{\abs{\sum_{0<\gamma\leq T}A(\gamma)}^{2}}{\sum_{0<\gamma\leq T}\abs{A(\gamma)}^{2}}\gg\frac{\abs{C}^{2}N(T)^{2}}{N(T)}=\abs{C}^{2} N(T).
        \end{equation}
    This proves that a positive proportion of the $A(\gamma)$'s are non-zero; in particular, for the same $\gamma$'s, $L(s,\chi_{1})$ and $L(s,\chi_{2})$ are linearly independent over the reals.
    \end{proof}

    \subsection{Proof of Proposition~\ref{prop1}}
    As the $d_{n}$'s contain only products of characters, we have $d_{n}=\bigo{1}$ (see proof of Proposition~\ref{prop3} for the calculation).
    In particular they define a multiplicative arithmetic function. We define, for a fixed $t$,
    \[B(s,P)\widesum{n\leq \qlt}\chi_{1}(n)n^{-s}=\widesum{n\leq R\qlt}d_{n}'n^{-s}.\] We have
    \begin{equation}
        d_{n} = \widesum{n=km}c_{k}\chi_{1}(m),
    \end{equation}
    and hence for $n\leq R\qlt$
    \begin{equation}\label{eq:dnprimedef}
        d_{n}' = \widesum{n=km\\ k\leq R}c_{k}\chi_{1}(m).
    \end{equation}
    From this it follows that $d_{n}=d_{n}'$ for $n\leq\qlt$. We also need to show that $d_{n}'\ll 1$.
    Let $p_{1},\dotsc,p_{h}$, for some $h>1$, denote all the primes below $P$ in an increasing order. Define $\widetilde{P}=p_{1}\dotsm p_{h}P$.
    From the product representation of $B(s,P)$, equation~\eqref{eq:bsp}, we see that $c_{n}=0$ for $n>1$, if $n$ contains any prime factors greater than $P$.
    This happens, in particular, if $(n,\widetilde{P})=1$. Hence for such $n$, $d_{n}'=\chi_{1}(n)$.
    Since $B(s,P)$ has a finite Euler product of degree two we have $c_{p^{j}}=0$ for any prime $p$ and $j\geq 3$.
    Thus we only need to consider $n>1$ with $n=p_{1}^{\alpha_{1}}\dotsm p_{h}^{\alpha_{h}}P^{\alpha_{0}}$, where $0\leq \alpha_{i}\leq 2$ for all $i$.
    The number of summands in~\eqref{eq:dnprimedef} is at most $(h+1)^{3}$. By~\eqref{eq:cestimate}, we find that $\abs{d_{n}'}\leq 2^{P}(h+1)^{3}$.
    In particular, $d_{n}'\ll 1$ as required.
    The approximate functional equation~\eqref{eq:lapprox} for $\chi_{1}$ with $\Delta=\sqrt{\ell}$ gives
    \begin{equation*}
        L(s, \chi_{1}) = \widesum{n\leq \qlt}\chi_{1}(n)n^{-s}+ X(s,\chi_{1})\widesum{n\leq\qtoverl}\overline{\chi}_{1}(n)n^{s-1}
        + \bigo{t^{-\sigma/2}\log t+t^{-1/4}},
    \end{equation*}
    where
    \begin{equation*}
        X(s,\chi)=\bige(\chi)\left(\frac{q}{\pi}\right)^{1/2-s}\frac{\Gamma\bigl(\frac{1-s+\af}{2}\bigr)}{\Gamma\left(\frac{s+\af}{2}\right)}.
    \end{equation*}
    Similarly for $\chi_{2}$ with $\Delta=\sqrt{q}R$ we get
    \begin{equation*}
        L(s, \chi_{2}) = \widesum{n\leq R\qlt}\chi_{2}(n)n^{-s}+ X(s,\chi_{2})\widesum{n\leq\frac{1}{R}\ltoverq}\overline{\chi}_{2}(n)n^{s-1} + \bigo{t^{-\sigma/2}\log{t}+t^{-1/4}}.
    \end{equation*}
    We can now expand the left-hand side in \eqref{prop1:eq1} to
    \begin{equation}\label{prop1:expanded}
        \begin{multlined}
            \widesum{0<\gamma\leq T}B(s,P)\\
            \times\Biggl(\rwidesum{n\leq\qlg}\chi_{1}(n)n^{-s}+ X(s,\chi_{1})\widesum{n\leq\qgoverl}\overline{\chi}_{1}(n)n^{s-1}+\bigo{\gamma^{-\sigma/2}\log \gamma+\gamma^{-1/4}}\Biggr)\\
            \times\Biggl(\rwidesum{n\leq R\qlg}\overline{\chi}_{2}(n)n^{-\overline{s}}+\overline{ X(s,\chi_{2})}\widesum{n\leq\frac{1}{R}\lgoverq}\chi_{2}(n)n^{\overline{s}-1}+\bigo{\gamma^{-\sigma/2}\log \gamma+\gamma^{-1/4}}\Biggr).
        \end{multlined}
    \end{equation}
    Denote the sum with $\chi_{1}$ and $\overline{\chi}_{2}$ by $M(T)$. We will take care of the other sums at the end of the proof.
    The main term comes from the diagonal entries of $M(T)$. First, write
    \begin{align}\label{prop1:maineq}
        M(T) &= \widesum{0<\gamma\leq T}B(s,P)\widesum{n\leq\qlg}\chi_{1}(n)n^{-s}\widesum{n\leq R\qlg}\overline{\chi}_{2}(n)n^{-\overline{s}}\notag\\
        &=\lwidesum{0<\gamma\leq T}\rwidesum{n\leq R\qlg}d_{n}'n^{-\sigma-i\gamma}\widesum{n\leq R\qlg}\overline{\chi}_{2}(n)n^{-\sigma+i\gamma}.\notag\\
        \intertext{Then, we separate the diagonal terms}
        M(T)&=\lwidesum{0<\gamma\leq T}\biggl(\rwidesum{n\leq R\qlg}\frac{d_{n}'\overline{\chi}_{2}(n)}{n^{2\sigma}}+\lwidesum{n\neq m}[R\qlg]\frac{d_{m}'\overline{\chi}_{2}(n)}{(nm)^{\sigma}}\left(\frac{n}{m}\right)^{i\gamma}\biggr)\notag\\ &= Z_{1}+Z_{2}.
    \end{align}
    The asymptotics in~\eqref{prop1:eq1} come from $Z_{1}$. We have
    \begin{align*}
        Z_{1} &=\lwidesum{0<\gamma\leq T}\biggl(\sum_{n=1}^{\infty}\frac{d_{n}\overline{\chi}_{2}(n)}{n^{2\sigma}}-\widesum{n>R\qlg}\frac{d_{n}\overline{\chi}_{2}(n)}{n^{2\sigma}}+\widesum{n\leq R\qlg}\frac{(d_{n}'-d_{n})\overline{\chi}_{2}(n)}{n^{2\sigma}}\biggr)\\
        &= N(T)\sum_{n=1}^{\infty}\frac{d_{n}\overline{\chi}_{2}(n)}{n^{2\sigma}}+C_{1}+C_{2}.
    \end{align*}
    We need to estimate $C_{1}$ and $C_{2}$. For $C_{1}$ we have
    \begin{equation*}
        C_{1} \ll \lwidesum{0<\gamma\leq T}\rwidesum{n>\sqrt{\gamma}}n^{-2\sigma}\ll\widesum{0<\gamma\leq T}\gamma^{1/2-\sigma}=\lilo{N(T)}.
    \end{equation*}
    Similarly,
    \begin{equation*}
        C_{2}\ll\lwidesum{0<\gamma\leq T}\rwidesum{n>\qlg}n^{-2\sigma}=\lilo{N(T)}.
    \end{equation*}
    To estimate $Z_{2}$ we wish to exchange the order of summation and apply the Gonek--Landau formula~\eqref{eq:landau}. Splitting and rewriting $Z_{2}$ in terms of the zeros of $\zeta$ we get
    \begin{align*}
        Z_{2} &=\lwidesum{0<\gamma\leq T}\sum_{n\leq R\qlg}\sum_{m<n}\left(\frac{d_{m}'\overline{\chi}_{2}(n)}{n^{\sigma+1/2}m^{\sigma-1/2}}\left(\frac{n}{m}\right)^{1/2+i\gamma}+\frac{d_{n}'\overline{\chi}_{2}(m)}{n^{\sigma+1/2}m^{\sigma-1/2}}\overline{\left(\frac{n}{m}\right)^{1/2+i\gamma}}\right)\\
        &=\lwidesum{n\leq R\qlT}\sum_{m<n}\sum_{\frac{2\pi n^{2}}{qlR^{2}}\leq\gamma\leq T}\left(\frac{d_{m}'\overline{\chi_{2}}(n)}{n^{\sigma+1/2}m^{\sigma-1/2}}\left(\frac{n}{m}\right)^{\rho}+\frac{d_{n}'\overline{\chi}_{2}(m)}{n^{\sigma+1/2}m^{\sigma-1/2}}\overline{\left(\frac{n}{m}\right)^{\rho}}\right).
    \end{align*}
    To apply the Gonek--Landau formula we split the innermost sum to $0<\gamma\leq T$ and $0<\gamma\leq 2\pi n^{2}/q\ell R^{2}$. Hence, we can write
    \[ Z_{2} = Z_{21} + Z_{22} + Z_{23} + Z_{24} + Z_{25},\]
    with
    \begin{align*}
        Z_{21} &= -\frac{T}{2\pi}\lwidesum{n\leq R\qlT}\sum_{m<n}\frac{d_{m}'\overline{\chi}_{2}(n)+d_{n}'\overline{\chi}_{2}(m)}{n^{\sigma+1/2}m^{\sigma-1/2}}\Lambda\left(\frac{n}{m}\right),\displaybreak[0]\\
        Z_{22} &\ll \sum_{n\leq R\qlT}\sum_{m<n}\frac{n^{2}\Lambda(n/m)}{n^{\sigma+1/2}m^{\sigma-1/2}},\displaybreak[0]\\
        Z_{23} &\ll \sum_{n\leq R\qlT}\sum_{m<n}\frac{1}{n^{\sigma+1/2}m^{\sigma-1/2}}\frac{n}{m}\log\frac{2nT}{m}\,\log\log\frac{3n}{m},\displaybreak[0]\\
        Z_{24} &\ll \sum_{n\leq R\qlT}\sum_{m<n}\frac{1}{n^{\sigma+1/2}m^{\sigma-1/2}}\log\frac{n}{m}\,\min\left(T,\frac{n/m}{\pdist{n/m}}\right),\displaybreak[0]\\
        \shortintertext{and}
        Z_{25} &\ll \log{T}\lwidesum{n\leq R\qlT}\sum_{m<n}\frac{1}{n^{\sigma+1/2}m^{\sigma-1/2}}\min\left(T,\frac{1}{\log(n/m)}\right).
    \end{align*}
    We begin by estimating $Z_{21}$. The only non-vanishing terms are with $m|n$. Thus we write $n=km$ and obtain
    \begin{equation*}
        Z_{21}\ll \frac{T}{2\pi}\lwidesum{k\leq R\qlT}\sum_{m<\frac{R}{k}\qlT}\frac{\Lambda(k)}{k^{\sigma+1/2}m^{2\sigma}}
        \ll\frac{T}{2\pi}\widesum{k\leq R\qlT}k^{\epsilon-\sigma-1/2}\widesum{m<\frac{R}{k}\qlT}m^{-2\sigma},
    \end{equation*}
    since $\Lambda(k)\ll k^{\epsilon}$ for any $\epsilon>0$. Since both sums are partial sums of convergent series we get $Z_{21}=\bigo{T}$.
    Working similarly with $Z_{22}$ gives
    \begin{align*}
        Z_{22} &\ll \lwidesum{k\leq R\qlT}\sum_{m<\frac{R}{k}\qlT}\frac{\Lambda(k)}{k^{\sigma-3/2}m^{2\sigma-2}}
        \ll \widesum{k\leq R\qlT}k^{3/2-\sigma+\epsilon}\widesum{m<\frac{R}{k}\qlT}m^{2-2\sigma}\\
        &\ll \widesum{k\leq R\qlT}k^{3/2-\sigma+\epsilon}\biggl(\biggl(\frac{T^{1/2}}{k}\biggr)^{3-2\sigma}+1\biggr)
        \ll T^{\frac{3-2\sigma}{2}}\widesum{k\ll T^{1/2}}k^{\sigma-3/2+\epsilon}=\bigo{T}.
    \end{align*}
    For $Z_{23}$ we get
    \begin{align*}
        Z_{23}&\ll\log T\,\log\log T\lwidesum{n\leq R\qlT}\frac{1}{n^{\sigma-1/2}}\sum_{m<n}\frac{1}{m^{\sigma+1/2}}\\
        &\ll\log T\,\log\log T\lwidesum{n\leq R\qlT}\frac{1}{n^{\sigma-1/2}}=\lilo{N(T)}.
    \end{align*}
    In order to estimate $Z_{24}$ we write $n=um+r$, where $-m/2<r\leq m/2$. Hence
    \begin{equation}\label{eq:disttoprime}
        \pdist*{u+\frac{r}{m}}=\begin{cases}\frac{\abs{r}}{m},&\text{if $u$ is a prime power and $r\neq 0$,}\\ \geq\frac{1}{2},&\text{otherwise.}\end{cases}
    \end{equation}
    Let $c=R\sqrt{q\ell/2\pi}$ then $n/m\leq n\leq c\sqrt{T}$, and so
    \begin{align*}
        Z_{24} &\ll\log T\lwidesum{n\leq cT^{1/2}}\sum_{m<n}\frac{1}{n^{\sigma+1/2}m^{\sigma-1/2}}\frac{n}{m}\frac{1}{\pdist{n/m}}\\
        &\ll\log T\lwidesum{m\leq cT^{1/2}}\sum_{u\leq\lfloor cT^{1/2}/m\rfloor+1}\sum_{-\frac{m}{2}<r\leq\frac{m}{2}}\frac{1}{m^{\sigma+1/2}(um+r)^{\sigma-1/2}}\frac{1}{\langle u+\frac{r}{m}\rangle},\\
        \intertext{and then evaluate the sum over $r$ depending on whether $u$ is a prime power or not to get}
        &\ll\log T\lwidesum{m\leq cT^{1/2}}\sum_{u\leq\lfloor cT^{1/2}/m\rfloor +1}\left(\Lambda(u)\frac{m\log m}{m^{\sigma+1/2}(um)^{\sigma-1/2}}+\frac{m}{m^{\sigma+1/2}(um)^{\sigma-1/2}}\right)\\
        &\ll\log T\lwidesum{m\leq cT^{1/2}}\frac{\log m}{m^{2\sigma-1}}\sum_{u\ll cT^{1/2}/m}\frac{u^{\epsilon}}{u^{\sigma-1/2}}=\bigo{T}.
    \end{align*}
    Finally, for $Z_{25}$ set $m=n-r$, $1\leq r\leq n-1$. So in particular $\log(n/m)>-\log(1-r/n)>r/n$. Hence,
    \begin{align*}
        Z_{25} &\ll\log T\lwidesum{n\leq cT^{1/2}}\sum_{1\leq r< n}\frac{1}{n^{\sigma+1/2}(n-r)^{\sigma-1/2}}\frac{n}{r}\\
        &\ll\log T\lwidesum{n\leq cT^{1/2}}\frac{1}{n^{\sigma-1/2}}\sum_{r\leq n-1}\frac{1}{r}=\bigo{T}.
    \end{align*}
    It remains to estimate all the other terms in~\eqref{prop1:expanded}. By repeating the analysis done for $Z_{1}$ and $Z_{2}$ we obtain the following estimates
    \begin{equation}\label{prop1:firstest1}
        \sum_{0<\gamma\leq T}\abs[\bigg]{\rwidesum{n\leq R\qlg}d_{n}'n^{-\sigma-i\gamma}}^{2}\ll N(T),
    \end{equation}
    and
    \begin{equation}\label{prop1:firstest2}
        \sum_{0<\gamma\leq T}\abs[\bigg]{\rwidesum{n\leq R\qlg}\overline{\chi}_{2}(n)n^{-\sigma+i\gamma}}^{2}\ll N(T).
    \end{equation}
    With trivial changes to the above argument we get,
    \begin{equation}\label{prop1:secest1}
        \sum_{0<\gamma\leq T}\abs[\bigg]{\rwidesum{n\leq\qgoverl}\chi_{1}(n)n^{\sigma-1+i\gamma}}^{2}\ll T^{\sigma-1/2}N(T),
    \end{equation}
    and
    \begin{equation}\label{prop1:secest2}
        \sum_{0<\gamma\leq T}\abs[\bigg]{\rwidesum{n\leq\frac{1}{R}\lgoverq}\overline{\chi}_{2}(n)n^{\sigma-1-i\gamma}}^{2}\ll T^{\sigma-1/2}N(T).
    \end{equation}
    We also need to estimate the order of growth of the derivative in $t$ of $\abs{X(s,\chi)}^{2}$.
    First, notice that $\abs{\bige(\chi)}=1$, so
    \begin{equation*}
        \abs{X(s,\chi)}=\left(\frac{q}{\pi}\right)^{1/2-\sigma}\abs*{\Gamma\left(\frac{1-s+\af}{2}\right)}\abs*{\Gamma\left(\frac{s+\af}{2}\right)}^{-1}.
    \end{equation*}
    By Stirling asymptotics
    \begin{equation}\label{eq:phiestimate}
        \abs{X(s,\chi)}^{2}\sim A\left(\frac{q}{\pi}\right)^{1-2\sigma}\gamma^{1-2\sigma},
    \end{equation}
    as $\overline{\Gamma(z)}=\Gamma(\overline{z})$, where $A$ is some non-zero constant.
    Thus, with $\psi=\Gamma'/\Gamma$,
    \begin{align}\label{prop1:diffestimate}\notag
        \frac{d}{d\gamma}\abs{X(s,\chi)}^{2} &=\abs{X(s,\chi)}^{2}\frac{i}{2}\biggl(\psi\biggl(\overline{\frac{1-s+\af}{2}}\biggr)-\psi\biggl(\frac{1-s+\af}{2}\biggr)+\psi\biggl(\overline{\frac{s+\af}{2}}\biggr)-\psi\biggl(\frac{s+\af}{2}\biggr)\biggr)\\ \notag
        &=\abs{X(s,\chi)}^{2}\frac{i}{2}\biggl(2i\biggl(\arg\biggl(\overline{\frac{1-s+\af}{2}}\biggr)-\arg\biggl(\frac{s+\af}{2}\biggr)\biggr)+\bigo{\gamma^{-2}}\biggr)\\
        &\ll\gamma^{1-2\sigma}\biggl(\bigo{\gamma^{-1}}+\bigo{\gamma^{-2}}\biggr)\ll\gamma^{-2\sigma},
    \end{align}
    by a standard estimate on $\psi$~\cite[p.~902,~8.361(3)]{gradshteyn2007} and the Taylor expansion of $\arccot$.
    Let
    \[ S(T) = \widesum{0<\gamma\leq T}\abs{X(\sigma+i\gamma,\chi_{1})}^{2}\abs[\bigg]{\sum_{n\leq R}c_{n}n^{-\sigma-i\gamma}}^{2}\abs[\bigg]{\rwidesum{n\leq\qgoverl}\chi_{1}(n)n^{\sigma-1+i\gamma}}^{2}.\]
    We use summation by parts, \eqref{prop1:diffestimate}, \eqref{prop1:secest1}, and \eqref{prop1:secest2} to see that
    \begin{multline*}
        S(T) =\abs{X(\sigma+iT,\chi_{1})}^{2}\lwidesum{0<\gamma\leq T}\abs[\bigg]{\rwidesum{n\leq\qgoverl}\chi_{1}(n)n^{\sigma-1+i\gamma}}^{2}\\
        -\int_{1}^{T}\lwidesum{0<\gamma\leq t}\abs[\bigg]{\rwidesum{n\leq\qgoverl}\chi_{1}(n)n^{\sigma-1+i\gamma}}^{2}\frac{d}{dt}\abs{X(\sigma+it,\chi_{1})}^{2}\d{t},
    \end{multline*}
    which simplifies to
    \[ S(T)\ll T^{1-2\sigma}T^{\sigma-1/2}N(T)+\int_{1}^{T}t^{\sigma-1/2}N(t)t^{-2\sigma}\d{t}.\]
    The first term is clearly $\lilo{N(T)}$. For the integral we use the fact that $N(t)=\bigo{t\log t}$ to estimate it as
    \begin{align*}
        \int_{1}^{T}t^{1/2-\sigma}\log t\d{t} &\ll T^{3/2-\sigma+\epsilon}.
    \end{align*}
    Hence we have that $S(T)=\lilo{N(T)}$, and similarly
    \begin{equation*}
        \widesum{0<\gamma\leq T}\abs*{X(\sigma+i\gamma,\overline{\chi}_{2})}^{2}\abs[\Bigg]{\rwidesum{n\leq\frac{1}{R}\lgoverq}\overline{\chi}_{2}(n)n^{\sigma-1-i\gamma}}^{2}\ll T^{3/2-\sigma+\epsilon}=\lilo{N(T)}.
    \end{equation*}
    Finally we use the Cauchy--Schwarz inequality, \eqref{prop1:firstest1}, \eqref{prop1:firstest2}, and the above two equations to estimate all other terms in \eqref{prop1:expanded} as $\lilo{N(T)}$.\qed

    \subsection{Proof of Proposition~\ref{prop2}}
    Since $B(s,P)$ is a finite Dirichlet polynomial it is bounded independently of $T$. Thus, to estimate $\sum_{\gamma\leq T}\abs{A(\gamma)}^{2}$, it suffices to estimate
    \begin{equation}\label{prop2:maineq}
        \widesum{0<\gamma\leq T}\abs*{L(s,\chi_{1})}^{2}\abs*{L(s,\chi_{2})}^{2}=\bigo{N(T)}.
    \end{equation}
    The approximate functional equation for $L(s,\chi_{1})$, as in the proof of Proposition~\ref{prop1}, gives
    \begin{align*}
        L(s,\chi_{1}) &=\widesum{n\leq\qlt}\chi_{1}(n)n^{-s}+ X(s,\chi_{1})\widesum{n\leq\qtoverl}\overline{\chi}_{1}(n)n^{s-1}+\bigo{t^{-\sigma/2}\log t+t^{-1/4}}\\
        &= W_{1}+ X(s,\chi_{1})W_{2}+\bigo{t^{-\sigma/2}\log t}+\bigo{t^{-1/4}}.\\
        \intertext{Similarly,}
        L(s,\chi_{2})&=Y_{1}+ X(s,\chi_{2})Y_{2}+\bigo{t^{-\sigma/2}\log t}+\bigo{t^{-1/4}},
    \end{align*}
    where \[Y_{1}=\widesum{n\leq\sqrt{\frac{qlt}{2\pi}}}\chi_{2}(n)n^{-s},\quad Y_{2}=\widesum{n\leq\sqrt{\frac{lt}{2\pi q}}}\overline{\chi}_{2}(n)n^{s-1}.\]
    We have
    \begin{equation}\label{prop2:mainsum}
        \widesum{0<\gamma\leq T}Y_{1}\overline{Y}_{1}W_{1}\overline{W}_{1}=\lwidesum{0<\gamma\leq T}\rwidesum{m,n,\mu,\nu\leq\qlg}\frac{\chi_{1}(m)\chi_{2}(n)\overline{\chi}_{1}(\mu)\overline{\chi}_{2}(\nu)}{(mn\mu\nu)^{\sigma}}\left(\frac{\mu\nu}{mn}\right)^{i\gamma}.
    \end{equation}
    Again, we consider the diagonal terms separately from the rest of the sum. The number of solutions to $mn=\mu\nu=r$ is at most the square of the number of divisors of $r$, $d(r)^{2}$.
    Thus
    \begin{equation}\label{prop2:mainterm}
        \lwidesum{0<\gamma\leq T}\rwidesum{mn=\mu\nu}[(q\gamma/2\pi)^{1/2}]\frac{\chi_{1}(m)\chi_{2}(n)\overline{\chi}_{1}(\mu)\overline{\chi}_{2}(\nu)}{(mn)^{2\sigma}}\ll\lwidesum{0<\gamma\leq T}\sum_{r=1}^{\infty}\frac{d(r)^{2}}{r^{2\sigma}}\ll N(T),
    \end{equation}
    %
    % CHANGE S TO W
    %
    since the inner series converges. For the off-diagonal terms set $\mu\nu=r$ and $mn=s$. We can treat the cases $s<r$ and $s>r$ separately. In the following analysis we assume $m$,~$n$,~$\mu$,~$\nu\leq (q\ell T/2\pi)^{1/2}$. Consider first the terms with $s<r$ in \eqref{prop2:mainsum}. We have that
    \begin{equation}
        Z_{2}=\lwidesum{r\leq\qlTs}\sum_{s<r}\rwidesum{m\mid s,\,\mu\mid r}\frac{\chi_{1}(m)\chi_{2}(s/m)\overline{\chi}_{1}(\mu)\overline{\chi}_{2}(r/\mu)}{r^{\sigma}s^{\sigma}}\lwidesum{K\leq\gamma\leq T}\left(\frac{r}{s}\right)^{i\gamma},
    \end{equation}
    where $K=\min(T,(2\pi/ql)\max(m^{2},s^{2}/m^{2},\mu^{2},r^{2}/\mu^{2}))$. Applying Gonek-Landau Formula~\eqref{thm:gonek} to $Z_{2}$ gives
    %
    % CHANGE Z_21 ETC TO SOMETHING ELSE
    %
    \begin{align*}
        Z_{2} &= \sum_{r\ll T}\sum_{s<r}\rwidesum{m\mid s,\,\mu\mid r}\frac{\chi_{1}(m)\chi_{2}(s/m)\overline{\chi}_{1}(\mu)\overline{\chi}_{2}(r/\mu)}{r^{\sigma+1/2}s^{\sigma-1/2}}\left(\sum_{0<\gamma\leq T}\left(\frac{r}{s}\right)^{\rho}-\lwidesum{0<\gamma<K}\left(\frac{r}{s}\right)^{\rho}\right)\\
        &=Z_{21,2}+Z_{23}+Z_{24}+Z_{25},
    \end{align*}
    with
    \begin{align*}
        Z_{21,2} &=\sum_{r\ll T}\sum_{s<r}\rwidesum{m\mid s,\,\mu\mid r}\frac{\chi_{1}(m)\chi_{2}(s/m)\overline{\chi}_{1}(\mu)\overline{\chi}_{2}(r/\mu)}{r^{\sigma+1/2}s^{\sigma-1/2}}\frac{K-T}{2\pi}\Lambda\left(\frac{r}{s}\right),\\
        Z_{23} &\ll \sum_{r\ll T}\sum_{s<r}\sum_{m\mid s,\,\mu\mid r}\frac{1}{r^{\sigma+1/2}s^{\sigma-1/2}}\frac{r}{s}\log\frac{2Tr}{s}\,\log\log\frac{3r}{s},\\
        Z_{24} &\ll \sum_{r\leq cT}\sum_{s<r}\rwidesum{m\mid s,\mu\mid r}\frac{\chi_{1}(m)\chi_{2}(s/m)\overline{\chi}_{1}(\mu)\overline{\chi}_{2}(r/\mu)}{r^{\sigma+1/2}s^{\sigma-1/2}}\log\frac{r}{s}\,\min\left(T,\frac{r/s}{\langle r/s\rangle}\right),\\
        \shortintertext{and}
        Z_{25} &\ll \sum_{r\ll T}\sum_{s<r}\sum_{m\mid s,\,\mu\mid r}\frac{1}{r^{\sigma+1/2}s^{\sigma-1/2}}\log 2T\,\min\left(T,\frac{1}{\log(r/s)}\right).
    \end{align*}
    For $Z_{21,2}$ we set $r=sk$. Since $d(x)\ll x^{\epsilon}$ and $K\leq T$, we get
    \begin{align*}
        Z_{21,2}&\ll T\sum_{k\ll T}\sum_{s\ll T/k}\frac{\Lambda(k)k^{\epsilon}s^{\epsilon}}{k^{\sigma+1/2}s^{2\sigma}}=\bigo{T}.
    \end{align*}
    We also have
    \begin{align*}
        Z_{23} &\ll \log T\,\log\log T\sum_{r\ll T}\frac{r^{\epsilon}}{r^{\sigma-1/2}}\sum_{s<r}\frac{s^{\epsilon}}{s^{\sigma+1/2}}\\
        &\ll \log T\,\log\log T\sum_{r\ll T}\frac{r^{\epsilon}}{r^{\sigma-1/2}}=\lilo{N(T)}.
    \end{align*}
    We can rewrite $Z_{24}$ as
    \begin{equation}\label{eq:ztwofour}
        \sum_{r\leq cT}\frac{\overline{(\chi_{1}\ast\chi_{2})}(r)}{r^{\sigma+1/2}}\sum_{s<r}\frac{(\chi_{1}\ast\chi_{2})(s)}{s^{\sigma-1/2}}\,\log\frac{r}{s}\,\min\left(T,\frac{r/s}{\langle r/s\rangle}\right),
    \end{equation}
    where $\ast$ denotes the Dirichlet convolution. Let $r=us+t$, where $-s/2<t\leq s/2$, and separate the terms where $u$ is not a prime power to $Z_{24,1}$, and denote the remaining terms by $Z_{24,2}$.
    We use \eqref{eq:disttoprime} to see that
    \begin{equation}
        Z_{24,1}\ll\sum_{s\leq cT}\sum_{u\ll cT/s+1}\sum_{\abs{t}<s/2}\frac{s^{\epsilon}(us+t)^{\epsilon}}{s^{\sigma+1/2}(us+t)^{\sigma-1/2}}\,\log \left(u+\frac{t}{s}\right).
    \end{equation}
    Rewriting yields
    \[Z_{24,1}\ll\log T\sum_{s\leq cT}\frac{s^{2\epsilon}}{s^{2\sigma}}\lwidesum{u\ll cT/s+1}\sum_{\abs{t}<s/2}\left(u+\frac{t}{s}\right)^{1/2-\sigma+\epsilon}.\]
    The terms in $u$ can be bound from above by $(u-1)^{1/2-\sigma+\epsilon}$. Thus
    \begin{align*}
        Z_{24,1} &\ll \log T\sum_{s\leq cT}s^{1+2\epsilon-2\sigma}\left(\frac{cT}{s}\right)^{3/2-\sigma+\epsilon}\\
        &\ll T^{3/2-\sigma+\epsilon}T^{1/2-\sigma+\epsilon}\log T=\bigo{N(T)}.
    \end{align*}
    For $Z_{24,2}$ let  $'$ in summation denote that the sum extends only over prime powers. We need to estimate
    \begin{equation*}
        \sum_{s\leq cT}\sideset{}{'}\sum_{u\leq\lfloor\frac{cT}{s}\rfloor+1}\rwidesum{0\neq\abs{t}<s/2}\frac{\overline{(\chi_{1}\ast\chi_{2})}(us+t)(\chi_{1}\ast\chi_{2})(s)}{(us+t)^{\sigma+1/2}s^{\sigma-1/2}}\log\left(u+\frac{t}{s}\right)\,\min\left(T,\frac{us+t}{\abs{t}}\right)
    \end{equation*}
    as $\bigo{N(T)}$. This can be rewritten as
    \begin{equation*}
        \sum_{s\leq cT}\frac{(\chi_{1}\ast\chi_{2})(s)}{s^{2\sigma}}\sideset{}{'}\sum_{u\leq\lfloor\frac{cT}{s}\rfloor+1}\rwidesum{0\neq\abs{t}<s/2}\frac{\overline{(\chi_{1}\ast\chi_{2})}(us+t)}{(u+\frac{t}{s})^{\sigma+1/2}}\log\left(u+\frac{t}{s}\right)\,\min\left(T,\frac{us+t}{\abs{t}}\right).
    \end{equation*}
    By taking absolute values and using the Triangle inequality we find that
    \begin{align*}
        Z_{24,2}&\ll \log^{2}T\,\sum_{s\ll T}s^{2\epsilon-2\sigma+1}\widesum{u\ll T/s}u^{1/2-\sigma+\epsilon}\\
        &\ll T^{3/2-\sigma+\epsilon}\log^{2}T\sum_{s\ll T}s^{\epsilon-\sigma-1/2}=\bigo{N(T)},
    \end{align*}
    as required.
    It remains to estimate $Z_{25}$. We use the same method as in Proposition~\ref{prop1}. Let $s=r-k$, and $1\leq k<r$ to get
    \begin{align*}
        Z_{25} &\ll \log T\sum_{r\ll T}\sum_{k<r}\frac{1}{r^{\sigma+1/2-\epsilon}(r-k)^{\sigma-1/2-\epsilon}}\frac{r}{k}\\
        &\ll \log T\sum_{r\ll T}\frac{1}{r^{\sigma-1/2-\epsilon}}\sum_{k< r}\frac{1}{k}=\lilo{N(T)}.
    \end{align*}
    Finally, if $s>r$ we can consider the complex conjugate of~\eqref{prop2:mainsum} to obtain the same estimate.
    The rest of the proof proceeds in the same way as in Proposition~\ref{prop1}. We obtain trivially the estimates
    \begin{align}
        \widesum{0<\gamma\leq T}\abs{W_{1}}^{4} &\ll N(T),\\
        \widesum{0<\gamma\leq T}\abs{Y_{1}}^{4} &\ll N(T).
    \end{align}
    Also, by modifying the argument slightly we find that
    \begin{align}
        \widesum{0<\gamma\leq T}\abs{W_{2}}^{4} &\ll T^{2\sigma-1+\epsilon}N(T),\\
        \widesum{0<\gamma\leq T}\abs{Y_{2}}^{4} &\ll T^{2\sigma-1+\epsilon}N(T).
    \end{align}
    %\begin{comment}
    %For example, if we let $S$ be the first sum then
    %\begin{align*}
    %S &\ll\sum_{\gamma\leq T}\sum_{m,n,\mu,\nu}\frac{\chi_{1}(m)\chi_{1}(n)\chi_{1}(\mu)\chi_{1}(\nu)}{(mn\mu\nu)^{1-\sigma}}\bigl(\frac{\mu\nu}{mn}\bigr)^{i\gamma}\\
    %&= S_{1}+S_{2},
    %\end{align*}
    %where $S_{1}$ is the sum over diagonal terms. We find
    %\begin{align*}
    %S_{1}&=\sum_{\gamma\leq T}\sum_{mn=\mu\nu}^{q\gamma/2\pi l}\frac{1}{(mn)^{2-2\sigma}}\\
    %&\ll \sum_{\gamma\leq T}\sum_{r=1}^{q\gamma/2\pi l}r^{2\sigma-2+\epsilon}\\
    %&\ll\sum_{\gamma\leq T}\gamma^{2\sigma-1+\epsilon}\\
    %&\ll T^{2\sigma-1+\epsilon}N(T).
    %\end{align*}
    %as required. We then proceed as before.
    %\end{comment}
    %\begin{comment}%We get T^{\sigma+1/2-\epsilon}T^{2\sigma-1+\epsilon}
    %For the off-diagonal terms we obtain after applying Landau's~Formula~\eqref{eq:landau}
    %\begin{align*}
    %S_{2} &=\sum_{r\ll T}\sum_{s<r}\sum_{m\mid s, \mu\mid r}\frac{\chi_{1}(s/m)\chi_{1}(m)\overline{\chi}_{1}(\mu)\overline{\chi}_{1}(r/\mu)}{r^{3/2-\sigma}s^{1/2-\sigma}}\frac{K-T}{2\pi}\lambda\biggl(\frac{r}{s}\biggr)\\
    %&+O\biggl(\sum_{r\ll T}\sum_{s<r}\sum_{m\mid s, \mu\mid r}\frac{1}{r^{1/2-\sigma}s^{3/2-\sigma}}\log T\log\log T\biggr)\\
    %&+O\biggl(\sum_{r\ll T}\sum_{s<r}\sum_{m\mid s, \mu\mid r}\frac{\chi_{1}(s/m)\chi_{1}(m)\overline{\chi}_{1}(\mu)\overline{\chi}_{1}(r/\mu)}{r^{3/2-\sigma}s^{1/2-\sigma}}\log T\frac{r/s}{\langle r/s\rangle}\biggr)\\
    %&+O\biggl(\sum_{r\ll T}\sum_{s<r}\sum_{m\mid s, \mu\mid r}\frac{1}{r^{3/2-\sigma}s^{1/2-\sigma}}\frac{\log T}{\log\sfrac{r}{s}}\biggr)\\
    %&= S_{21,2}+S_{23}+S_{24}+S_{25}.
    %\end{align*}
    %\end{comment}
    We also need to estimate the derivative of $\abs{X(s,\chi)}^{4}$. By estimate~\eqref{prop1:diffestimate} from Proposition~\ref{prop1} we get
    \begin{align*}
        \frac{d}{d\gamma}\abs{X(s,\chi)}^{4}&\ll\gamma^{1-2\sigma}\gamma^{-2\sigma}\ll\gamma^{1-4\sigma}.
    \end{align*}
    The rest of the proof now follows from estimating
    \[\widesum{0<\gamma\leq T}\abs{X(s,\chi_{1})}^{4}\abs{W_{2}}^{4}=\bigo{T^{1-2\sigma+\epsilon}N(T)},\]
    and similarly for $Y_{2}$, and applying the Cauchy--Schwarz to the remaining terms in the expansion of the product in~\eqref{prop2:maineq}.\qed

    \subsection{Proof of Proposition~\ref{prop3}}
    Let
    \[D=\sum_{n=1}^{\infty}\frac{d_{n}\overline{\chi}_{2}(n)}{n^{2\sigma}},\quad E=\sum_{n=1}^{\infty}\frac{e_{n}\overline{\chi}_{1}(n)}{n^{2\sigma}}.\]
    By Proposition~\ref{prop2} it is sufficient to show that $D-E\neq 0$.
    First, we need to compute the $d_{n}$ and $e_{n}$'s explicitly. Let us denote the set of primes smaller than $P$ by $\primesP=\{p_{1},p_{2},\dotsc,p_{h}, P\}$.
    Suppose $P$ is large enough so that $q$, $\ell\in\primesP$. The coefficients $d_{n}$ are defined by the Euler product
    \begin{equation*}
        \prod_{p\leq P}(1-\chi_{2}(p)p^{-s})\times\prod_{p>P}\sum_{n=0}^{\infty}\frac{\chi_{1}(p^{n})}{p^{ns}}.
    \end{equation*}
    If $p^{2}\mid n$, $p\in\primesP$, then $n$ disappears from the expansion, i.e.~$d_{n}=0$. If $n$ has no prime factors from the set $\primesP$, then we just get the usual coefficient from $L(s,\chi_{1})$. On the other hand, if some prime $p\in\primesP$ divides $n$ exactly once then it contributes $-\chi_{2}(p)$. Hence
    \begin{equation*}
        d_{n} =\begin{cases}\chi_{1}(n),&\text{if $p\nmid n$ for all $p\in\primesP$,}\\ (-1)^{k}\chi_{1}\bigl(\frac{n}{p_{i_{1}}\dotsm p_{i_{k}}}\bigr)\chi_{2}(p_{i_{1}}\dotsm p_{i_{k}}),&\text{if $p_{i_{j}}\parallel n$ for $p_{i_{j}}\in\primesP$ for all $j$,}\\ 0&\text{otherwise.}\end{cases}
    \end{equation*}
    %
    % MAYBE WRITE FORMULA FOR E_N
    %
    Similarly for $e_{n}$. Hence for $p>P$ the Euler factors of $D$ are of the form
    \begin{equation*}
        (1-\chi_{1}(p)\overline{\chi}_{2}(p)p^{-2\sigma})^{-1},
    \end{equation*}
    while for $E$ one obtains the complex conjugate. On the other hand, for $p\leq P$ we have
    \begin{equation*}
        1+d_{p}\overline{\chi}_{2}(p)p^{-2\sigma}+d_{p^{2}}\overline{\chi}_{2}(p^{2})p^{-4\sigma}+\dotsb = 1 - \chi_{2}\overline{\chi}_{2}(p)p^{-2\sigma}=1-p^{-2\sigma},
    \end{equation*}
    unless $p=\ell$, and similarly for the second series. Now, suppose that $D=E$, then
    \begin{equation*}
        \prod_{\substack{p\leq P\\p\neq\ell}}(1-p^{-2\sigma})\prod_{p>P}(1-(\chi_{1}\overline{\chi}_{2})(p)p^{-2\sigma})^{-1} =
        \prod_{\substack{p\leq P\\p\neq q}}  (1-p^{-2\sigma})\prod_{p>P}(1-(\overline{\chi}_{1}\chi_{2})(p)p^{-2\sigma})^{-1}.
    \end{equation*}
    We cancel out the common terms in the product over $p<P$, which yields
    \[(1-q^{-2\sigma})z=(1-\ell^{-2\sigma})\bar{z},\]
    where
    \[z=\prod_{p>P}(1-(\chi_{1}\overline{\chi}_{2})(p)p^{-2\sigma})^{-1}.\]
    Hence,
    \[\frac{1-q^{-2\sigma}}{1-\ell^{-2\sigma}}=\frac{\bar{z}}{z}.\]
    Taking absolute values yields
    \[\frac{1-q^{-2\sigma}}{1-\ell^{-2\sigma}}=1,\]
    which is a contradiction.\qed

    % % % % % % % % % % % % % % % % % % % % %
    %										%
    %	§4	Proof of Thm 2					%
    %										%
    \section{Proof of Theorem~\ref{thmtwo}} %
    % % % % % % % % % % % % % % % % % % % % %

    We now sample the values of $L(s,\chi)$ at precisely the non-trivial zeros of $\zeta$.
    In this case we do not assume RH. Off the critical line we used the method of Gonek--Landau to prove linear independence.
    On the critical line, however, this becomes very difficult.
    This is mainly because of the corresponding $Z_{24}$ term in the first proposition. We get
    \[Z_{24}=\sum_{n\leq X}\sum_{m<n}\frac{\chi(m/n)}{n}\log\left(\frac{n}{m}\right)\min\left(T,\frac{n/m}{\langle n/m\rangle}\right),\]
    where $X=qT/2\pi\sqrt{\log T}$. This should be $\lilo{N(T)\log T}$, which seems to be very difficult to prove.
    In the proof of Conrey et al.~\cite{conrey1986} they make a reduction to the discrete mean values of one $L$-function at a time.
    We have been unable to find such a reduction in our case.
    Garunkštis et~al.~\cite{garunkstis2010} presented a more suitable method through contour integration and a modified Gonek Lemma (see Lemma~\ref{modgonek}).

    Denote the characters in Theorem~\ref{thmtwo} by $\chi_{1}$ and $\chi_{2}$ with distinct prime moduli $q$ and $\ell$.
    For any Dirichlet character $\mu$ modulo $n$, we denote the principal character modulo $n$ by $\mu_{0}$.
    Moreover, put $B(s,p)=p^{s}$ for some prime $p$ to be determined later.
    Then, \[A(\gamma):=B(\rho,p)\left(L(\rho,\chi_{1})-L(\rho,\chi_{2})\right)\] is non-zero
    precisely when the two $L$-functions assume distinct values.
    \begin{prop}\label{prop21}
        Let $\ccontour$ be the rectangular contour with vertices at $a+i$, $a+iT$, $1-a+iT$, and $1-a+i$ with
        positive orientation, where $a=1+(\log T)^{-1}$. Then we have
        \begin{equation}\label{eq:prop21_1}
            \frac{1}{2\pi i}\int_{\ccontour}\frac{\zeta'}{\zeta}(s)B(s,p)L(s,\chi_{1})\d{s} \sim \frac{\overline{C}_{\chi_{1}}T}{2\pi}\log \frac{T}{2\pi},
        \end{equation}
        where
        \[ C_{\chi_{1}} = \frac{G(1,\overline{\chi}_{1})G(-p,\chi_{1})}{q},\]
        and similarly for $\chi_{2}$.
    \end{prop}
    Then, by the residue theorem, we get
    \[\widesum{0<\gamma\leq T}A(\gamma)=\frac{1}{2\pi i}\int_{\ccontour}\frac{\zeta'}{\zeta}(s)B(s,p)\left(L(s,\chi_{1})-L(s,\chi_{2})\right)\d{s}.\]
    \begin{prop}\label{prop22}
        With the same contour as in Proposition~\ref{prop21} we have for $j,j'\in\{1,2\}$ that
        \[ \frac{1}{2\pi i}\int_{\ccontour}\frac{\zeta'}{\zeta}(s)L(s,\chi_{j})L(1-s,\overline{\chi}_{j'})\d{s}\ll T\log^{2}T. \]
    \end{prop}
    This proposition gives us estimates for all the terms in $\sum\abs{A(\gamma)}^{2}$, since $B(s,p)$ can be bounded independently of $T$.
    Finally, we have to prove that the difference coming from Proposition~\ref{prop21} is non-zero.
    \begin{prop}\label{prop23}
        There is a prime $p$, different from $q$ and $\ell$, such that
        $C_{\chi_{1}}-C_{\chi_{2}}\neq 0$.
    \end{prop}
    With these propositions we can prove Theorem~\ref{thmtwo} in the same way as in~\eqref{thmoneproof}.
    In the proofs below we make extensive use of the following facts about Gauss sums, see~\cite[pg.~65~(2)]{davenport1980}.
    \begin{equation}\label{eq:gauss1}
        G(n,\chi_{1})=\overline{\chi_{1}}(n)G(1,\chi_{1}),
    \end{equation}
    since $q$ is a prime, and (see~\cite[pg.~66~(5)]{davenport1980})
    \begin{equation}\label{eq:gauss2}
        G(1,\overline{\chi_{1}})G(-1,\chi_{1})=q,
    \end{equation}
    and similarly for $\chi_{2}$.

    \subsection{Proof of Proposition~\ref{prop21}}
    We prove the proposition for $\chi_{1}$ as the case of $\chi_{2}$ is identical.
    Denote the integral in~\eqref{eq:prop21_1} by $\inti$. Then
    \begin{align*}
        \inti &= \left(\int_{a+i}^{a+iT}+\int_{a+iT}^{1-a+iT}+\int_{1-a+iT}^{1-a+i}+\int_{1-a+i}^{a+i}\right)\frac{\zeta'}{\zeta}(s)B(s,p)L(s,\chi_{1})\d{s}\\
        &= \inti_{1} + \inti_{2} + \inti_{3} + \inti_{4}.
    \end{align*}
    We can evaluate $\inti_{1}$ explicitly to get
    \begin{align*}
        \inti_{1}	&= \int_{a+i}^{a+iT}\frac{\zeta'}{\zeta}(s)L(s,\chi_{1})p^{s}\d{s}\\
        &= -i\sum_{n,\,m}\frac{\Lambda(n)\chi_{1}(m)}{(mnp^{-1})^{a}}\int_{1}^{T}\left(\frac{p}{mn}\right)^{it}\d{t}\ll \frac{\zeta'}{\zeta}(a)\zeta(a)+T=\bigo{T},
    \end{align*}
    where the second term comes from the case $mn=p$.
    For $\inti_{2}$ we use the following bounds (see~\cite[pg.~108]{davenport1980}):
    \begin{equation}\label{eq:zetapbound}
        \frac{\zeta'}{\zeta}(\sigma+iT)\ll\log^{2}T,\quad\text{if } -1\leq\sigma\leq 2,\quad \abs{T}\geq 1,
    \end{equation}
    and
    \begin{equation}\label{eq:lbound}
        L(\sigma+iT,\chi_{1})\ll\abs{T}^{1/2}\log\abs{T+2},\quad\text{if } 1-a\leq \sigma\leq a, \quad \abs{T}\geq 1.
    \end{equation}
    These yield $\inti_{2}=\bigo{T^{1/2}\log^{3}T}$.
    Next we consider $\inti_{3}$. Changing variables $s\mapsto 1-\bar{s}$ gives
    \begin{equation*}
        \inti_{3} = \frac{-1}{2\pi i}\int_{a+i}^{a+iT}\frac{\zeta'}{\zeta}(1-\bar{s})L(1-\bar{s},\chi_{1})p^{1-\bar{s}}\d{s}.
    \end{equation*}
    Conjugating and applying the functional equation of $\zeta$ and $L(s,\chi_{1})$ yields
    \begin{equation*}
        \overline{\inti}_{3} = \frac{p}{2\pi i}\int_{a+i}^{a+iT}\left(\frac{\zeta'}{\zeta}(s)+\frac{\gamma'}{\gamma}(s)\right)L(s,\chi_{1})\Delta(s,\chi_{1})p^{-s}\d{s},
    \end{equation*}
    where
    \[\gamma(s) = \pi^{1/2-s}\frac{\Gamma\left(\frac{s}{2}\right)}{\Gamma\left(\frac{1-s}{2}\right)},\]
    and
    \[\Delta(s,\chi_{1})=\left(\frac{q}{2\pi}\right)^{s}\frac{1}{q}G(1,\overline{\chi}_{1})\Gamma(s)\left(e^{-\pi is/2}+\overline{\chi}_{1}(-1)e^{\pi is/2}\right).\]
    Using the definition of $\Delta$ to expand the above we find that
    \[\overline{\inti}_{3}= p\left(\intf_{1}+\dotsb+\intf_{4}\right),\]
    where
    \begin{align*}
        \intf_{1} &= \frac{G(1,\overline{\chi}_{1})}{q}\frac{1}{2\pi i}\int_{a+i}^{a+iT}\frac{\gamma'}{\gamma}(s)\left(\frac{q}{2\pi p}\right)^{s}\Gamma(s)\exp\left(-\frac{\pi is}{2}\right)L(s,\chi_{1})\d{s},\\
        \intf_{2} &= \frac{\overline{\chi}_{1}(-1)G(1,\overline{\chi}_{1})}{q}\frac{1}{2\pi i}\int_{a+i}^{a+iT}\frac{\gamma'}{\gamma}(s)\left(\frac{q}{2\pi p}\right)^{s}\Gamma(s)\exp\left(+\frac{\pi is}{2}\right)L(s,\chi_{1})\d{s},\\
        \intf_{3} &= \frac{G(1,\overline{\chi}_{1})}{q}\frac{1}{2\pi i}\int_{a+i}^{a+iT}\frac{\zeta'}{\zeta}(s)\left(\frac{q}{2\pi p}\right)^{s}\Gamma(s)\exp\left(-\frac{\pi is}{2}\right)L(s,\chi_{1})\d{s},\\
        \intf_{4} &= \frac{\overline{\chi}_{1}(-1)G(1,\overline{\chi}_{1})}{q}\frac{1}{2\pi i}\int_{a+i}^{a+iT}\frac{\zeta'}{\zeta}(s)\left(\frac{q}{2\pi p}\right)^{s}\Gamma(s)\exp\left(+\frac{\pi is}{2}\right)L(s,\chi_{1})\d{s}.
    \end{align*}
    We rewrite $\intf_{1}$ in the following way
    \begin{equation}\label{eq:f1parts}
        \intf_{1} = \frac{G(1,\overline{\chi}_{1})}{q}\int_{1}^{T}\frac{\gamma'}{\gamma}(a+i\tau)\d{\left(\frac{1}{2\pi i}\int_{a+i}^{a+i\tau}\left(\frac{q}{2\pi p}\right)^{s}\Gamma(s)\exp\left(-\frac{\pi is}{2}\right)L(s,\chi_{1})\d{s}\right)}.
    \end{equation}
    By Lemma~\ref{modgonek},
    \begin{equation*}
        \frac{1}{2\pi i}\int_{a+i}^{a+i\tau}\left(\frac{q}{2\pi p}\right)^{s}\Gamma(s)\exp\left(-\frac{\pi i s}{2}\right)L(s,\chi_{1})\d{s}
        =\widesum{n\leq\frac{\tau q}{2\pi p}}\chi_{1}(n)\exp\left(-2\pi i\frac{np}{q}\right)+\bigo{\tau^{1/2+\epsilon}}.
    \end{equation*}
    We can separate the periods to write the sum as
    \[\sum_{a=1}^{q}\chi_{1}(a)\exp\left(-2\pi i\frac{ap}{q}\right)\widesum{\substack{n\leq\frac{\tau q}{2\pi p}\\ n\equiv a\bmod{q}}}1=\frac{\tau}{2\pi p}G(-p,\chi_{1})+O(1).\]
    We integrate by parts in~\eqref{eq:f1parts} and use the standard estimate
    \[ \frac{\gamma'}{\gamma}(s) = \log\frac{\abs{t}}{2\pi}+\bigo{\abs{t}^{-1}},\quad\abs{t}\geq 1, \]
    to see that
    \begin{align*}
        \intf_{1}	&= \frac{C_{\chi_{1}}}{2\pi p}\int_{1}^{T}\left(\log\frac{\tau}{2\pi}+\bigo{\tau^{-1}}\right)\d{\left(\tau+\bigo{\tau^{1/2+\epsilon}}\right)}\\
        &= \frac{C_{\chi_{1}}T}{2\pi p}\log \frac{T}{2\pi}+\bigo{T^{1/2+\epsilon}}.
    \end{align*}
    Similarly by Lemma~\ref{modgonek}, $\intf_{2}$ is $\bigo{\log T}$, while $\intf_{4}=\bigo{1}$.
    For $\intf_{3}$ we have
    \begin{align*}
        \intf_{3}	&= \frac{G(1,\overline{\chi}_{1})}{q}\frac{1}{2\pi i}\int_{a+i}^{a+iT}\left(\frac{q}{2\pi p}\right)^{s}\Gamma(s)\exp\left(-\frac{\pi is}{2}\right)\frac{\zeta'}{\zeta}(s)L(s,\chi_{1})\d{s}\\
        &=\frac{-G(1,\overline{\chi}_{1})}{q}\widesum{mn\leq \frac{Tq}{2\pi p}}\Lambda(m)\chi_{1}(n)\exp\left(-2\pi i\frac{mnp}{q}\right)+\bigo{T^{1/2+\epsilon}}.
    \end{align*}
    Looking at the summation we decompose it as
    \[\widesum{m\leq\frac{Tq}{2\pi p}}\Lambda(m)\widesum{n\leq\frac{Tq}{2\pi pm}}\chi_{1}(n)\exp\left(-2\pi i\frac{mnp}{q}\right).\]
    We separate the periods in the same way as for $\intf_{1}$ and write the above sum as
    \[\widesum{m\leq\frac{Tq}{2\pi p}}\Lambda(m)G(-mp,\chi_{1})\frac{T}{2\pi pm}+\bigo{T}.\]
    We will show that the summation over $m$ in fact converges. This means that we have $\intf_{3}=\bigo{T}$.
    To do this it suffices to consider
    \[\sum_{m\leq X}\frac{\Lambda(m)\overline{\chi}_{1}(m)}{m}.\]
    Let
    \[\psi(X,\overline{\chi}_{1})=\sum_{m\leq X}\Lambda(m)\overline{\chi}_{1}(m).\]
    Then, by~\cite[pg.~123~(8)]{davenport1980},
    \[\psi(X,\overline{\chi}_{1})=-\frac{X^{\beta}}{\beta}+\bigo{X\exp(-c(\log X)^{1/2})},\]
    where the term with $\beta$ comes from the Siegel zero of $\chi_{1}$ and $c$ is some positive absolute constant. However, since our $q$ is fixed, we know that
    $\beta$ is bounded away from 1. Hence, with summation by parts we obtain
    \[\sum_{m\leq X}\frac{\Lambda(m)\overline{\chi}_{1}(m)}{m}=\frac{\psi(X,\overline{\chi}_{1})}{X}+\int_{1}^{X}\frac{\psi(t,\overline{\chi}_{1})}{t^{2}}\d{t}=\bigo{1}\]
    as required.
    Finally $\inti_{4}=\bigo{1}$ as the integrand is analytic in a neighbourhood of the line of integration.\qed

    \subsection{Proof of Proposition~\ref{prop22}}
    We prove the case $j=j'=1$ as the other cases are either similar or easier. Now, denote the integral by $\inti$, i.e.
    \[ \inti = \frac{1}{2\pi i}\int_{\ccontour}\frac{\zeta'}{\zeta}(s)L(s,\chi_{1})L(1-s,\overline{\chi}_{1})\d{s},\]
    and split it in the same way as in the proof of Proposition~\ref{prop21}, so that
    \[ \inti = \inti_{1}+\dotsb+\inti_{4}.\]
    We can write $\inti_{1}$ as
    \begin{align*}
        \inti_{1}	&= \frac{1}{2\pi i}\int_{a+i}^{a+iT}\frac{\zeta'}{\zeta}(s)L(s,\chi_{1})^{2}\Delta(s,\chi_{1})\d{s}\\
        &= \frac{G(\overline{\chi}_{1})}{q}\frac{1}{2\pi i}\int_{a+i}^{a+iT}\left(\frac{q}{2\pi}\right)^{s}\Gamma(s)\left(\exp\left(\frac{-\pi is}{2}\right)+\overline{\chi}_{1}(-1)\exp\left(\frac{\pi is}{2}\right)\right)\frac{\zeta'}{\zeta}(s)L(s,\chi_{1})^{2}\d{s}\\
        &= \inte_{1}+\inte_{2}.
    \end{align*}
    By Lemma~\ref{modgonek}, $\inte_{2}=\bigo{1}$. Let us now estimate $\inte_{1}$. We have
    \[ \inte_{1}=-\frac{G(\overline{\chi}_{1})}{q}\widesum{mn\leq\frac{Tq}{2\pi}}d(n)\chi_{1}(n)\Lambda(m)\exp\left(-2\pi i\frac{nm}{q}\right)+\bigo{T^{1/2+\epsilon}}.\]
    Denote the sum over $m$ and $n$ by $S$. As before, we first separate the periods
    \begin{align*}
        S &= \widesum{a,\,b=1}[q]\chi_{1}(a)\exp\left(-2\pi i\frac{ab}{q}\right)\widesum{\substack{mn\leq\frac{Tq}{2\pi}\\ n\equiv a\bmod{q}\\ m\equiv b\bmod{q}}}d(n)\Lambda(m).\\
        \intertext{Now sum over characters $\eta$ of modulus $q$ to get}
        &=\frac{1}{\varphi(q)}\sum_{\eta\bmod{q}}\widesum{a,\,b=1}[q]\chi_{1}(a)\overline{\eta}(a)\exp\left(-2\pi i\frac{ab}{q}\right)\widesum{\substack{mn\leq\frac{Tq}{2\pi}}\\ m\equiv b\bmod{q}}d(n)\eta(n)\Lambda(m)\\
        &=\frac{1}{\varphi(q)}\sum_{\eta\bmod{q}}\sum_{b=1}^{q}G(-b,\chi_{1}\overline{\eta})\widesum{\substack{mn\leq\frac{Tq}{2\pi}}\\ m\equiv b\bmod{q}}d(n)\eta(n)\Lambda(m),\\
        \shortintertext{and as before}
        &=\frac{1}{\varphi(q)^{2}}\widesum{\eta,\,\omega\bmod{q}}G(-1,\chi_{1}\overline{\eta})\left(\rwidesum{\substack{mn\leq\frac{Tq}{2\pi}}}d(n)\eta(n)\Lambda(m)\omega(m)\right)\sum_{b=1}^{q}\overline{\chi}_{1}(b)\eta(b)\overline{\omega}(b).
    \end{align*}
    The sum over $b$ is non-zero if and only if $\omega=\omega_{0}$ and $\eta=\chi_{1}$; or $\omega=\overline{\chi_{1}}$ and $\eta=\eta_{0}$; or $\omega\neq\omega_{0}$ and $\eta=\chi_{1}\omega$.
    By Perron's formula
    \[-\widesum{mn\leq\frac{Tq}{2\pi}}d(n)\eta(n)\Lambda(m)\omega(m)=\frac{1}{2\pi i}\int_{a-iU}^{a+iU}\frac{L'}{L}(s,\omega)L(s,\eta)^{2}\left(\frac{Tq}{2\pi}\right)^{s}\frac{ds}{s}+\bigos{\frac{T\log^{3}T}{U}}\]
    for some $U$ with $|U|\leq T$.
    Since our characters are fixed, we can use Vinogradov-type zero-free region~\cite[pg.~296]{garunkstis2010}. That is, let $b_{1}=1-c_{1}/(\log t)^{3/4+\epsilon}$
    (in fact, any power smaller than 1 would do), then
    $L(\sigma+it,\chi_{1})$ has no zeros in the region $\sigma\geq b_{1}$.
    Here $c_{1}$ is some positive absolute constant. By the approximate functional equation and Stirling asymptotics we have uniformly for $0<\sigma< 1$ and $\abs{t}>1$ that
    \begin{equation}\label{eq:convexity}
        L(s,\chi_{1})\ll\abs{t}^{\frac{1-\sigma}{2}}\log(\abs{t}+1).
    \end{equation}
    Then, by shifting the contour we get
    \begin{multline}\label{eq:perron}
        -\widesum{mn\leq\frac{Tq}{2\pi}}d(n)\eta(n)\Lambda(m)\omega(m)=\Res_{s=1}\frac{L'}{L}(s,\omega)L(s,\eta)^{2}\left(\frac{Tq}{2\pi}\right)^{s}\frac{1}{s}\\
        -\frac{1}{2\pi i}\left(\int_{a+iU}^{b_{1}+iU}+\int_{b_{1}+iU}^{b_{1}-iU}+\int_{b_{1}-iU}^{a-iU}\right)\frac{L'}{L}(s,\omega)L(s,\eta)^{2}\left(\frac{Tq}{2\pi}\right)^{s}\frac{ds}{s}+\bigos{\frac{T\log^{3}T}{U}}.
    \end{multline}
    We need to find the residues in each of the three cases.
    \begin{align*}
        \Res_{s=1}\frac{L'}{L}(s,\omega_{0})L(s,\chi_{1})^{2}\left(\frac{Tq}{2\pi}\right)^{s}\frac{1}{s}&=-L(1,\chi_{1})^{2}\frac{Tq}{2\pi},\\
        \Res_{s=1}\frac{L'}{L}(s,\overline{\chi}_{1})L(s,\eta_{0})^{2}\left(\frac{Tq}{2\pi}\right)^{s}\frac{1}{s}&=\frac{Tq}{2\pi}\left(\frac{\varphi(q)}{q}\right)^{2}\frac{L'}{L}(1,\overline{\chi}_{1})\log\frac{Tq}{2\pi}+\bigo{T},\\
        \Res_{s=1}\frac{L'}{L}(s,\omega)L(s,\chi_{1}\omega)^{2}\left(\frac{Tq}{2\pi}\right)^{s}\frac{1}{s} &= 0.
    \end{align*}
    It remains to estimate the integrals on the right-hand side of~\eqref{eq:perron}.
    By~\eqref{eq:zetapbound} and~\eqref{eq:convexity} we see that the first and third integrals yield $\bigo{T^{a}U^{-b_{1}}\log^{4}U}$.
    We split the second integral and estimate it as
    \begin{multline*}
        \left(\int_{b_{1}+iU}^{b_{1}+i}+\int_{b_{1}+i}^{b_{1}-i}+\int_{b_{1}-i}^{b_{1}-iU}\right)\frac{L'}{L}(s,\overline{\chi}_{1})L(s,\eta_{0})^{2}\left(\frac{Tq}{2\pi}\right)^{s}\frac{ds}{s}\\
        =\bigo{T^{b_{1}}U^{1-b_{1}}\log^{4}U+T^{b_{1}}\abs{b_{1}-1}^{-2}},
    \end{multline*}
    where the second error term comes from the integral over the constant segment.
    It suffices to choose $U=T^{1/2}$ as then
    \begin{align*}
        T^{a}U^{-b_{1}}\log^{4}U &\ll Te^{-\frac{1}{2}\log T+\frac{c_{1}}{2}(\log T)^{1/4-\epsilon}+4\log\log T},\\
        T^{b_{1}}U^{1-b_{1}}\log^{4}U&\ll Te^{-\frac{c_{1}}{2}(\log T)^{1/4-\epsilon}+4\log\log T},\\
        \shortintertext{and}
        T^{b_{1}}|b_{1}-1|^{-2} &\ll Te^{-c_{1}(\log T)^{1/4-\epsilon}+(3/2+2\epsilon)\log\log T},
    \end{align*}
    which are all $\bigo{T}$.
    Therefore we conclude that $\inti_{1}=\bigo{T\log T}$.

    Next up is $\inti_{2}$. We use again the convexity bound writing it as
    \[ L(\sigma+it,\chi_{1})\ll_{\epsilon}\abs{t}^{\mu_{0}(\sigma)+\epsilon},\]
    where $\epsilon>0$, $-1<\sigma<2$ (say), $\abs{t}>1$, and
    \[ \mu_{0}(\sigma)=\begin{cases} 0,& \text{if } \sigma >1,\\ \frac{1-\sigma}{2}, & \text{if } 0<\sigma<1,\\ \frac{1}{2}-\sigma, & \text{if } \sigma<0.\end{cases}\]
    With this we can write
    \begin{align*}
        \inti_{2}	&= \frac{1}{2\pi i}\int_{a+iT}^{1-a+iT}\frac{\zeta'}{\zeta}(s)L(s,\chi_{1})L(1-s,\overline{\chi}_{1})\d{s}\\
        &\ll \left(\int_{1-a}^{0}+\int_{0}^{1}+\int_{1}^{a}\right)\log^{2}T\, T^{\mu_{0}(\sigma)+\epsilon}T^{\mu_{0}(1-\sigma)+\epsilon}\d{\sigma}.
    \end{align*}
    Keeping in mind that $\sigma\leq a$ we get
    \[ \inti_{2}\ll T^{a-1/2+\epsilon}\log^{2}T+T^{1/2+\epsilon}\log^{2}T=\bigo{T}. \]
    For $\inti_{3}$ we do the usual trick of mapping $s\mapsto 1-\bar{s}$. Taking complex conjugates leads to
    \[ \overline{\inti}_{3} = \frac{1}{2\pi i}\int_{a+i}^{a+iT}\left(\frac{\zeta'}{\zeta}(s)+\frac{\gamma'}{\gamma}(s)\right)L(s,\chi_{1})^{2}\Delta(s,\chi_{1})\d{s}. \]
    As in Proposition~\ref{prop21} we split this up into $\intf_{1}, \dotsc, \intf_{4}$. Adding up $\intf_{3}$ and $\intf_{4}$ gives $\inti_{1}$, which is $\bigo{T\log T}$.
    As before, $\intf_{2}$ does not contribute.
    So we have to estimate $\intf_{1}$, that is
    \[\intf_{1}=\frac{G(1,\overline{\chi}_{1})}{q}\int_{1}^{T}\left(\log\frac{\tau}{2\pi}+\bigo{\tau^{-1}}\right)\,\d{\left(\frac{1}{2\pi i}\int_{a+i}^{a+i\tau}L(s,\chi_{1})^{2}\Gamma(s)\exp\left(-\frac{\pi is}{2}\right)\d{s}\right)}.\]
    Working as in Proposition~\ref{prop21} we can write the inner integral (plus an error term) as
    \[ \widesum{n\leq\frac{\tau q}{2\pi}}\chi_{1}(n)d(n)\exp\left(-2\pi i\frac{n}{q}\right). \]
    This is $\bigo{\tau\log\tau}$, which gives $\inti_{3}=\bigo{T\log^{2}T}$. It is not difficult to extend this to an asymptotic estimate, but for our purposes
    the upper bound is sufficient.
    Trivially we also have that $\inti_{4}=\bigo{1}$. Hence $\inti=\bigo{T\log^{2} T}$.\qed

    \subsection{Proof of Proposition~\ref{prop23}}
    By~\eqref{eq:gauss1} and~\eqref{eq:gauss2} we see that
    \[C_{\chi_{1}}=C_{\chi_{2}}\]
    if and only if $\chi_{1}(p)=\chi_{2}(p)$. By Chinese Remainder Theorem and Dirichlet's Theorem we can find a prime $p$ different from $q$ and $\ell$ that satisfies
    \[\left\{\begin{aligned}
                p &\equiv 1\bmod{q},\\
                p &\equiv a\bmod{\ell},
            \end{aligned}\right.\]
    such that $\chi_{2}(p)=\chi_{2}(a)\neq 1$, since $\chi_{2}$ is non-principal. This gives $1=\chi_{1}(p)=\chi_{2}(p)\neq 1$, which is a contradiction.
    \qed
    % % % % % % % % % % % % % % % % % % % % %
    %										%
    %	§5	Bibliography					%
    %										%
    % % % % % % % % % % % % % % % % % % % % %
    %\section{References}

    %\addcontentsline{toc}{section}{References}
    %\printbibliography[title=References]
    \printbibliography

\end{document}